\documentclass[a4paper, 12pt]{article}
\usepackage[applemac]{inputenc}
\usepackage{amsmath, amssymb, amsfonts, amsthm}
\usepackage{fancybox,xspace}
\usepackage{fullpage}
\theoremstyle{plain}
\newtheorem{theorem}{Theorem}[section]
\newtheorem{lemma}[theorem]{Lemma}
\newtheorem{proposition}[theorem]{Proposition}
\newtheorem{corollary}[theorem]{Corollary}

\theoremstyle{definition}
\newtheorem{definition}[theorem]{Definition}
\newtheorem{example}[theorem]{Example}

\theoremstyle{remark}
\newtheorem{rem}[theorem]{Remark}

\newcommand{\sM}{\mathcal{M}}
\newcommand{\fC}{\mathfrak{C}}

\newcommand{\DD}{\mathbb{D}}
\newcommand{\QQ}{\mathbb{Q}}
\newcommand{\R}{\mathbb{R}}
\renewcommand{\SS}{\mathbb{S}}
\newcommand{\XX}{\mathbb{X}}

\newcommand{\eps}{\varepsilon}
\newcommand{\one}{\mathbf{1}}
\newcommand{\thf}{\frac{1}{2}}

\newcommand{\Ind}{\mathbf{1}}
\renewcommand{\P}{\mathbf{P}}
\newcommand{\E}{\mathbf{E}}
\newcommand{\eqd}{\stackrel{\scriptscriptstyle{\mathcal{D}}}{\sim}}
\newcommand{\iid}{i.\,i.\,d.\xspace}
\newcommand{\as}{a.\,s.\xspace}

\newcommand{\salg}{\mathfrak{F}}

\newcommand{\cadlag}{c\`adl\`ag\xspace}
\newcommand{\ie}{i.\,e.\xspace}
\newcommand{\eg}{e.\,g.\xspace}

\begin{document}

\title{Series Representation of Time-Stable Stochastic Processes}

\author{Christoph Kopp \and\ Ilya Molchanov}

\maketitle

\begin{abstract}
  A stochastically continuous process $\xi(t)$, $t\geq0$, is said to
  be time-stable if the sum of $n$ \iid copies of $\xi$ equals in
  distribution to the time-scaled stochastic process $\xi(nt)$,
  $t\geq0$. The paper advances the understanding of time-stable
  processes by means of their LePage series representations.
\end{abstract}

\section{Introduction}

The (strict) stability property of stochastic processes is
conventionally defined by requiring that the sum of \iid copies of a
process is distributed as the scaled variant of the original
process. An alternative scaling operation applied to the time argument
leads to another definition of stability.

\begin{definition}
  \label{def1}
  A stochastically continuous real-valued process $\xi(t)$, $t\geq0$,
  is said to be \emph{time-stable} if, for each $n\geq 2$,
  \begin{equation}
    \label{eq:ts-def}
    \xi_1+\cdots+\xi_n \eqd n\circ\xi\,,
  \end{equation}
  where $\xi_1,\dots,\xi_n$ are \iid copies of $\xi$, $\eqd$ means the
  equality of all finite-dimensional distributions and
  $(n\circ\xi)(t)=\xi(nt)$, $t\geq0$, is the process obtained by time
  scaling $\xi$.
\end{definition}

Definition~\ref{def1} goes back to Mansuy~\cite{RM}, where processes
satisfying \eqref{eq:ts-def}, regardless of their stochastic
continuity, are called \emph{infinitely divisible with respect to
  time} (IDT). Indeed, they are infinitely divisible in the sense that
$\xi$ can be represented as the sum of $n$ \iid processes for each
$n\geq2$. However, the time-stable name better emphasises the
particular stability feature of such processes. These processes have
been recently investigated in \cite{EO} and \cite{hak:ouk13}, also
with a multivariate time argument. Time-stable processes with values
in $\R^d$ can be defined similarly to Definition~\ref{def1}.

The major difficulty in the analysis of time-stable processes stems
from the necessity to work with the whole paths of the processes. The
time stability concept cannot be formulated in terms of
finite-dimensional distributions at any given time moments, since the
time argument in the right-hand side of \eqref{eq:ts-def} is scaled.

Definition~\ref{def1} can be modified to introduce
$\alpha$-time-stable processes as
\begin{displaymath}
  \xi_1+\cdots+\xi_n \eqd n^{1/\alpha}\circ\xi\,,
\end{displaymath}
where each $\alpha\neq 0$ is possible. This concept appears in
\cite[Ex.~8.12]{DMZ} as an example of the stability property in the
cone of continuous functions with scaling applied to the
argument. While such processes (for general $\alpha$) have been
considered in \cite{hak:ouk12a}, the process $\xi(t^{1/\alpha})$,
$t\geq0$, obtained by time change is time-stable (with $\alpha=1$) and
so it is not necessary to study $\alpha$-time stability for general
$\alpha\neq1$.

Another closely related concept is that of a \emph{dilatively stable
  process} $\zeta$ that satisfies the following scaling relation for
some $\alpha>0$, $\delta\in(0,2\alpha]$, and all $n\geq2$:
\begin{displaymath}
  \zeta_1+\cdots+\zeta_n \eqd n^{1/2-\alpha/\delta}(n^{1/\delta}\circ\zeta)\,,
\end{displaymath}
see \cite{igl:bar12}, where such processes are also assumed to possess
moments of all orders. If $\zeta$ is dilatively stable, then
$\xi(t)=t^{1/2-\alpha/\delta}\zeta(t^{1/\delta})$, $t\geq0$, satisfies
\eqref{eq:ts-def} and so is a time-stable process if it is
stochastically continuous.

Barczy et al. \cite{bar:ker:pap14} defined
\emph{$(\rho_1,\rho_2)$-aggregate self-similar processes} $\zeta$ by the
following scaling relation
\begin{displaymath}
  \zeta_1+\cdots+\zeta_n \eqd n^{\rho_1}(n^{-\rho_2}\circ\zeta)\,,
\end{displaymath}
so that for $\rho_1=\rho_2$ one recovers the aggregate similar process
from \cite{kaj05}. It is easy to see that
$t^{\rho_1}\zeta(t^{-\rho_2})$, $t\geq0$, satisfies \eqref{eq:ts-def},
so that this and other above mentioned generalisations may be obtained
by time and scale change from time-stable processes.  A similar
concept was studied by Penrose~\cite{pen92}, who called a stochastic
process $\xi$ \emph{semi-min-stable} if $\min(\xi_1(t),\dots,\xi_n(t))$
shares the finite-dimensional distributions with $n^{-1}\xi(n^\alpha
t)$, $t\geq0$.

Section~\ref{sec:exampl-basic-prop} discusses elementary properties
and examples of time-stable processes. The infinite divisibility of
such processes makes it possible to use their spectral representation
obtained in \cite{kab:stoev14} and then show that the L\'evy measure
is homogeneous with respect to time scaling, see
Section~\ref{sec:spectr-repr}. The main result of
Section~\ref{sec:lepage-series} and of the whole paper is the LePage
representation of time-stable processes whose L\'evy measures are
supported by the family of right-continuous functions with left
limits. In particular, this is the case for non-negative
processes. The obtained LePage representation yields the series
representations for dilatively stable and aggregate self-similar
processes. The structure of pure jump time-stable processes is closely
related to the stability property of marked point processes; in this
case the LePage representation is similar to the cluster
representation of infinitely divisible point processes, see
Section~\ref{sec:pure-jump-time}.

The concept of time stability allows generalisations in various
directions. The necessary structure consists of a time set which is
invariant under scaling by arbitrary positive real numbers and an
associative and commutative binary operation which is applied
pointwisely to the values of processes. For instance, the definition
applies also to stochastic processes defined on the whole line and on
$\R^d$ or with addition replaced by the maximum operation.

While \eqref{eq:ts-def} actually defines a \emph{strictly} time-stable
stochastic process, the stability concept can be relaxed by replacing
the right hand side with $n\circ\xi+f_n$ for deterministic functions
$\{f_n\}$. Furthermore, it is possible to consider random measures
stable with respect to scaling of their argument (see
\cite[Ex.~8.23]{DMZ}) and generalised stochastic processes (\ie random
generalised functions).

\section{Elementary properties and examples} 
\label{sec:exampl-basic-prop}

The following result provides an alternative definition of time-stable
processes.

\begin{proposition}
  \label{prop:equiv-ts}
  A stochastically continuous process $\xi(t)$, $t\geq0$, is
  time-stable if and only if
  \begin{equation}
    \label{eq:ab-tst}
    a\circ\xi_1+b\circ\xi_2\eqd (a+b)\circ\xi
  \end{equation}
  for all $a,b>0$, where $\xi_1$ and $\xi_2$ are independent copies of
  $\xi$.
\end{proposition}
\begin{proof}
  The sufficiency is evident. The necessity is first shown for
  positive rational $a,b$ and then for all positive $a,b$ by
  approximation using the stochastic continuity assumption.
\end{proof}

Each L\'evy process is time-stable. If $\xi$ is time-stable, then it
is infinitely divisible, so that there always exists the unique L\'evy
process $\tilde{\xi}$, called the \emph{associated L\'evy process} of
$\xi$, such that $\tilde{\xi}(1)$ coincides in distribution with
$\xi(1)$. Since $\xi$ is stochastically continuous, $\tilde{\xi}(t)$
coincides in distribution with $\xi(t)$ for each given $t$ (see
\cite[Prop.~4.1]{RM}), so that the characteristic function of $\xi(t)$
is given by
\begin{equation}
  \label{timestab.char.eq}
  \E\exp\{\imath\lambda\xi(t)\}
  = \exp\{-t\Psi(\lambda)\}\,,\quad t\geq0\,,
\end{equation}
where $\Psi$ denotes the cumulant of $\tilde{\xi}$. In view of this,
$\Psi$ is also called the cumulant of $\xi$. The higher order
finite-dimensional distributions of $\xi$ and $\tilde{\xi}$ may
differ.

It follows from \eqref{timestab.char.eq} that $\xi(t)$ weakly
converges to $0$ as $t\downarrow0$, which corresponds to the
stochastic continuity of $\xi$, since $\xi(0)=0$ \as by
\eqref{eq:ts-def}. The only \as constant time stable process is the
zero process.

Comparing the one-dimensional distributions shows that if the
non-degenerate time-stable process is \as non-negative for any
$t\geq0$, then it is \as non-negative everywhere, its one-dimensional
distributions are increasing in the stochastic order, and
$\sup_{t\geq0}\xi(t)$ is \as infinite. In contrast to L\'evy
processes, non-negative time-stable processes need not be
a.\,s. monotone, see Example~\ref{example:scaled}.

\begin{theorem}
  \label{thr:rep}
  A time-stable process $\xi$ is identically distributed as the sum of
  a linear function, a centred Gaussian process with the covariance
  function $C(t,s)$ that satisfies $C(ut,us)=uC(t,s)$ for all
  $t,s,u\geq0$, and an independent time-stable process without
  Gaussian component.
\end{theorem}
\begin{proof}
  Since $\xi$ is infinitely divisible, its finite-dimensional
  distributions are infinitely divisible. The rest follows by
  comparing the L\'evy triplets of the $n$-fold convolution of
  $(\xi(t_1),\dots,\xi(t_k))$ and of $(\xi(nt_1),\dots,\xi(nt_k))$ for
  any $t_1,\dots,t_k\geq 0$ and $k,n\geq1$.
\end{proof}

Various characterisations of Gaussian time-stable processes are
presented in \cite{RM}.
In the following we mostly consider time-stable processes without
a Gaussian part.

\begin{example}[cf. \cite{RM}]
  \label{example:scaled}
  Let $\xi$ be a time-stable process. For $c_1,\dots,c_m\in \R$ and
  $s_1,\dots,s_m\geq 0$, 
  \begin{displaymath}
    \tilde{\xi}(t)=\sum_{i=1}^m c_i \xi(ts_i),\quad t\geq 0\,, 
  \end{displaymath}
  is a time-stable process.
  For instance, if $N(t)$, $t\geq0$, is a homogeneous Poisson process,
  then $N(2t)-N(t)$, $t\geq0$, is a non-monotonic a.s. non-negative
  time-stable process with the standard Poisson process as its
  associated L\'evy process.
\end{example}

\begin{example}[cf. \cite{RM}]
  \label{stastex2}
  Let $\zeta$ be a strictly $\alpha$-stable random variable with
  $\alpha\in(0,2]$. Then $\xi(t)=t^{1/\alpha}\zeta$, $t\geq0$, is a
  time-stable stochastic process. In particular, if $\zeta$ is
  normally distributed, then $\xi(t)=\sqrt{t}\zeta$ is time-stable. If
  $\alpha=1$, then $\xi(t)=\zeta t$ for the Cauchy random variable
  $\zeta$. This yields a time-stable process with stationary
  increments, which is not a L\'evy process. This construction can be
  generalised by letting
  \begin{equation}
    \label{eq:stastex2}
    \xi(t)=\sum_{i=1}^m c_i t^{1/\alpha_i}\zeta_i\,,\quad t\geq0\,,
  \end{equation}
  for constants $c_1,\dots,c_m$ and independent random variables
  $\zeta_1,\dots,\zeta_m$ that are strictly stable with parameters
  $\alpha_1,\dots,\alpha_m$. The independence condition can be relaxed
  if the vector $\zeta=(\zeta_1,\dots,\zeta_m)$ is multi-stable meaning
  that the sum of its $n$ \iid copies is distributed as
  $(n^{1/\alpha_1}\zeta_1,\dots,n^{1/\alpha_m}\zeta_m)$ for every $n\geq2$.
\end{example} 

The following construction is similar to the construction of
sub-stable random elements, see \cite[Sec.~1.3]{ST}.

\begin{example}[Sub-stable processes] 
  \label{ex:sub-stable}
  Consider the process $X(t)=\xi(t^{1/\alpha}\zeta)$, $t\geq0$, where
  $\xi$ is time-stable and $\zeta$ is an independent of $\xi$ positive
  strictly $\alpha$-stable random variable with $\alpha\in(0,1)$.
  Conditioning on \iid copies $\zeta_1,\zeta_2$ of $\zeta$ and using
  Proposition~\ref{prop:equiv-ts}, we obtain that for \iid copies
  $X_1,X_2$ of $X$ and any $a,b>0$
  \begin{align*}
    a\circ X_1(t) + b\circ X_2(t)
    &=\xi_1(a^{1/\alpha}\zeta_1t^{1/\alpha})+\xi_2(b^{1/\alpha}\zeta_2t^{1/\alpha})\\
    &\eqd \xi((a^{1/\alpha}\zeta_1+b^{1/\alpha}\zeta_2)\,t^{1/\alpha})\\
    &\eqd \xi((a+b)^{1/\alpha}\zeta t^{1/\alpha})
    =(a+b)\circ X(t)\,,\quad t\geq 0\,,
  \end{align*}
  showing the time stability of $X$.
\end{example}

\begin{example}[Subordination by time-stable processes]
  \label{ex:subord}
  Let $X$ be a L\'evy process and let $\xi$ be an \as non-decreasing
  independent of $X$ time-stable process. Then $X(\xi(t))$ is also
  time-stable. This fact was proved using characteristic functions in
  \cite[Th.~3.6]{EO}.  The famous variance gamma process used in
  finance (see, e.\,g. \cite{mad:car:chan98}) appears if $X=W$ is the
  Brownian motion and $\xi(t)$ is the gamma subordinator. The
  construction from Example~\ref{example:scaled} suggests further
  variants of the variance gamma process, \eg $W(\xi(t)+\xi(2t))$.
\end{example}

A stochastic process $\xi$ is called \emph{$H$-self-similar} if for
any $c>0$
\begin{equation}
  \label{eq:selfsim-def}
  c\circ\xi \eqd c^{H} \xi,
\end{equation}
where $H>0$ is called the \emph{Hurst index} or self-similarity
index. 
A time-stable process is strictly $\alpha$-stable if and only
if it is $1/\alpha$-self-similar, see \cite[Prop.~3.1]{EO}.
By \cite[Th.~7.5.4]{ST}, any time-stable process, which is also
$\alpha$-stable with $\alpha\in(0,1)$ and has stationary increments,
is a L\'evy process.

\begin{example}[Sub-Gaussian time-stable process]
  \label{ex:sub-g}
  Each Gaussian time-stable process is $\thf$-self-similar. Therefore,
  a fractional Brownian motion $\eta(t)$, $t\geq 0$, with Hurst index
  $H\neq\thf$ is not time-stable. Since a Gaussian self-similar
  process with stationary increments is necessarily a fractional
  Brownian motion, see \cite[Th.~1.3.3]{EM}, the only Gaussian
  time-stable process with stationary increments is the Brownian
  motion. It is possible to obtain a sub-Gaussian time-stable process
  as
  \begin{displaymath}
    \xi(t)=Z^{1/2}\eta(t^{1/(2\alpha H)}),\qquad t\geq0\,,
  \end{displaymath}
  where $\alpha\in(0,1)$ and $Z$ is an \as positive $\alpha$-stable
  random variable independent of $\eta$. In particular, if
  $H=1/(2\alpha)$ for $\alpha\in[1/2,1)$, then $\xi(t)=Z^{1/2}\eta(t)$
  and if $H=\thf$, then $\xi(t)=Z^{1/2}\eta(t^{1/\alpha})$ for the
  Brownian motion $\eta$.
\end{example}

\section{L\'evy measures of time-stable processes}
\label{sec:spectr-repr}

The time stability property of the process $\xi$ implies its infinite
divisibility.  Each stochastically continuous process satisfies
Condition S from \cite{ST}, which requires the existence of an at most
countable set $T_0\subset[0,\infty)$ such that for all $t\geq0$, there
exists a sequence $t_n\in T_0$, $n\geq1$, such that $\xi(t_n)$
converges to $\xi(t)$ in probability.  The spectral representation of
infinitely divisible stochastic processes that satisfy Condition S and
do not possess a Gaussian component is obtained in
\cite[Th.~2.14]{kab:stoev14} using a Poisson process on a certain
space $(\Omega,\salg)$ with a $\sigma$-finite measure
$\mu$. Reformulating this result for $(\Omega,\salg)$ being the space
$\R^{[0,\infty)}$ with the cylindrical $\sigma$-algebra $\fC$ yields
that an infinitely divisible stochastically continuous process $\xi$
admits a spectral representation
\begin{equation}
  \label{eq:spectr-int}
  \xi(t)\eqd c(t) +\int f(t) d\Pi_Q(f)\,,
\end{equation}
where $c$ is a deterministic function and $\Pi_Q=\{f_i(t):\; i\geq
1\}$ is the Poisson process on $\R^{[0,\infty)}\setminus\{0\}$ with a
$\sigma$-finite intensity measure $Q$ such that
\begin{equation}
  \label{eq:mu-condition}
  \int_{\R^{[0,\infty)}\setminus\{0\}} \min(1,f(t)^2)Q(df)<\infty
\end{equation}
for all $t\geq0$. The measure $Q$ is called the \emph{L\'evy measure}
of $\xi$. The integral with respect to $\Pi_Q$ in
\eqref{eq:spectr-int} is defined as the \as existing limit of the compensated
sums
\begin{equation}
  \label{eq:compensated-sum}
  \lim_{r\downarrow 0} \left[\sum_{f_i\in \Pi_Q}
  f_i(t)\one_{|f_i(t)|>r}
  -\int_{\{f:\;|f(t)|>r\}} L(f(t))Q(df)\right]\,,
\end{equation}
where 
\begin{equation}
\label{eq:levy.fun}
  L(u)=
  \begin{cases}
    u, & |u|\leq 1\,,\\
    1, & u>1\,,\\
    -1, & u<-1\,,
  \end{cases}
\end{equation}
is a L\'evy function, see also \cite{mar70}. 

Furthermore, \cite[Th.~2.14]{kab:stoev14} ensures the existence of a
\emph{minimal} spectral representation, meaning that the
$\sigma$-algebra generated by $\{f:\; f(t)\in A\}$ for all $t\geq 0$
and Borel $A\subset\R$ coincides with the cylindrical $\sigma$-algebra
$\fC$ on $\R^{[0,\infty)}$ up to $Q$-null sets and there is no set $B
\in \fC$ with $Q(B)>0$ such that for every $t\geq0$, $Q(\{f\in B:\;
f(t)\neq 0\})=0$. In the following assume that the $\sigma$-algebra on
$\R^{[0,\infty)}$ is complete with respect to $Q$. By
\cite[Th.~2.17]{kab:stoev14}), the minimal spectral representation is
unique up to an isomorphism, and so the L\'evy measure is well
defined. 

The stochastic continuity of $\xi$ yields that $\xi$ has a measurable
modification, see \cite[Th.~3.3.1]{GS}. Then
\cite[Prop.~2.19]{kab:stoev14} establishes that the representation
\eqref{eq:spectr-int} involves a measurable function $c(t)$, $t\geq0$,
and that the functions $f$ can be chosen to be strongly separable. The
latter means that there exists a measurable null-set
$\Omega_0\subset\R^{[0,\infty)}$ and a countable set
$\QQ\subset[0,\infty)$ (called a separant) such that, for each open
$G\subset [0,\infty)$ and closed $F\subset\R$, we have
\begin{equation}
  \label{eq:strongly-sep}
  \{f:\; f(t)\in F\;\forall t\in G\cap\QQ\}\setminus
  \{f:\; f(t)\in F\; \forall t\in G\}
  \subset \Omega_0\,.
\end{equation}

The process $\xi$ is symmetric, i.\,e. $\xi$ coincides in distribution
with $-\xi$, if and only if its L\'evy measure $Q$ is symmetric
and $c$ identically vanishes. Then the compensating second term
in \eqref{eq:compensated-sum} vanishes for every $r>0$.

If \eqref{eq:mu-condition} is strengthened to require
\begin{equation}
  \label{eq:mu-condition-1}
  \int_{\R^{[0,\infty)}\setminus\{0\}} \min(1,|f(t)|)Q(df)<\infty\,,
\end{equation}
then the integral \eqref{eq:spectr-int} is well defined without taking the
limit and without the compensating term in \eqref{eq:compensated-sum},
so that
\begin{equation}
  \label{eq:4}
  \xi(t)\eqd c(t)+\sum_{f_i \in \Pi_Q} f_i(t)
\end{equation}
for a deterministic function $c$. It is well known that
\eqref{eq:mu-condition-1} holds if $\xi(t)$ is \as non-negative for
all $t\geq0$, see \eg \cite[Th.~51.1]{KS}.

Representation \eqref{eq:4} shows that if the process is monotone
(resp. non-negative), then the measure $Q$ is supported by monotone
(resp. non-negative) functions.

\begin{lemma}
  \label{lemma:measurab}
  For each Borel $B\subset\R^{[0,\infty)}$ and $s>0$, the set $s\circ
  B=\{s\circ f:\; f\in B\}$ is also Borel in $\R^{[0,\infty)}$.
\end{lemma}
\begin{proof}
  The set $s\circ B$ is Borel if $B$ is a cylinder, and the statement
  follows from the monotone class argument.
\end{proof}
The next result follows from the fact that $\xi(0)=0$ \as for a
time-stable process $\xi$.

\begin{lemma}
  \label{lem:zero}
  The L\'evy measure of a time-stable process is supported by
  $\{f\in\R^{[0,\infty)}\setminus\{0\}:\; f(0)=0\}$. 
\end{lemma}

\begin{lemma}
  \label{lemma:scaling}
  An infinitely divisible stochastically continuous process $\xi$
  without a Gaussian component is time-stable if and only if
  $c(t)=bt$, $t\geq0$, is a linear function for a constant $b\in\R$
  and the L\'evy measure $Q$ satisfies
  \begin{equation}
    \label{eq:q-scales}
    Q(s\circ B)=s^{-1}Q(B),\qquad s>0\,,
  \end{equation}
  for all Borel $B\subset\R^{[0,\infty)}$. 
\end{lemma}
\begin{proof}
  The sufficiency follows because the finite-dimensional
  distributions of $\xi$ admit the characteristic function
  \begin{multline}
    \label{eq:pgf-Q}
    \E \exp\{\imath \sum_j \theta_j \xi(t_j)\}\\
    =\exp\left\{\imath b\sum_j\theta_jt_j+\int\left[e^{\imath \sum_j \theta_j
          f(t_j)}-1 -\imath \sum_j \theta_j L(f(t_j))\right]
      Q(df)\right\}\,.
  \end{multline}
  Now assume that $\xi$ is time-stable.  Comparing the characteristic
  functions of the finite-dimensional distributions for the processes
  in the left and right-hand side of \eqref{eq:ab-tst} shows that the
  function $c$ is linear.

  The L\'evy measure corresponding to the minimal spectral
  representation of the process in the left-hand side of
  \eqref{eq:ab-tst} is $Q(a^{-1}\circ B)+Q(b^{-1}\circ B)$. In view of
  the uniqueness of the minimal spectral representation
  \cite[Th.~2.17]{kab:stoev14}, the L\'evy measures of the processes
  in the left and right-hand sides of \eqref{eq:ab-tst} coincide.
  Thus
  \begin{displaymath}
    Q(a^{-1}\circ B)+Q(b^{-1}\circ B)=Q((a+b)^{-1}\circ B)
  \end{displaymath}
  for all $a,b>0$ and all Borel $B\subset\R^{[0,\infty)}$. Given that
  $Q$ is non-negative, \cite[Th.~1.1.7]{bin:gol:teu87} yields that
  $Q(s^{-1}\circ B)$ is a linear function of $s$,
  i.\,e. \eqref{eq:q-scales} holds.
\end{proof}

The same scaling property for the L\'evy measure appears in
\cite[Lemma~5.1]{RM} and later on was reproduced in
\cite[Prop.~4.1]{hak:ouk12a} for time-stable processes with paths in
the Skorokhod space of right-continuous functions with left limits
(\cadlag functions). The proof there is however incomplete, since it
is not shown that the L\'evy measure of such a process is supported by
\cadlag functions.

\begin{proposition}
  \label{prop:d}
  If $\xi(t)$, $t\geq0$, is a time-stable \cadlag process with
  \as non-negative values, then its L\'evy measure $Q$ is supported
  by \cadlag functions. 
\end{proposition}
\begin{proof}
  In this case the L\'evy measure $Q$ satisfies
  \eqref{eq:mu-condition-1} and so $\xi$ admits the representation
  \eqref{eq:4}. If $\xi'$ is an independent copy of $\xi$, then
  $\xi-\xi'$ is symmetric and has the series decomposition with the
  L\'evy measure supported by \cadlag functions, see
  \cite[Th.~4]{ros89p}. The support of $Q$ is a subset of the support
  of the L\'evy measure for $\xi-\xi'$.
\end{proof}

\section{ LePage series representation}
\label{sec:lepage-series}

In finite-dimensional spaces, L\'evy measures of strictly stable laws
admit a polar decomposition into the product of a radial and a finite
directional part and the corresponding sum (if necessary compensated)
of points of the Poisson process is known as the LePage series, see
\cite[Cor.~3.10.4]{ST} and \cite{lep:wood:zin81,rosi90}. The LePage
series can be defined in general spaces \cite{DMZ}, where they provide
a rich source of stable laws and in many cases characterise stable
laws. However, the general conditions of \cite{DMZ} do not hold for
the case of time-stable processes. 

The following result provides the LePage series characterisation for
time-stable processes without a Gaussian part and whose L\'evy measure
is supported by the family $\DD'$ of not identically vanishing \cadlag
functions on $[0,\infty)$. We endow the family $\DD'$ with the
topology and the $\sigma$-algebra induced from $\R^{[0,\infty)}$. Let
$\DD'_0$ be the family of not identically vanishing \cadlag functions
that vanish at the origin.

\begin{theorem}
  \label{thr:lp-stab}
  The following statements are equivalent for a stochastically
  continuous process $\xi(t)$, $t\geq0$.
  \begin{itemize}
  \item[i)] The process $\xi$ is time-stable without a Gaussian part
    and with the L\'evy measure $Q$ supported by $\DD'$.
  \item[ii)] The stochastic process $\xi$ is infinitely divisible
    without a Gaussian part, with a deterministic linear part, its
    L\'evy measure $Q$ is supported by $\DD'_0$, and
    \begin{equation}
      \label{eq:q-lp}
      Q(B)=\int_0^\infty \sigma(t\circ B) dt
    \end{equation}
    for all Borel sets $B\subset \DD'_0$ and a probability measure
    $\sigma$ on $\DD'_0$ such that
    \begin{equation}
      \label{eq:5}
      \int_{\DD'_0}
      \int_0^\infty \min(1,f(t)^2)t^{-2}dt\sigma(df)<\infty\,.
    \end{equation}
  \item[iii)] The stochastic process $\xi$ has the same distribution
    as 
    \begin{multline}
      \label{eq:lepage-compensated}
      ct+ \lim_{r\downarrow 0} \Bigg[\sum_{i=1}^\infty
      \eps_i(\Gamma_i^{-1}t)\one_{|\eps_i(\Gamma_i^{-1}t)|>r}\\
      -\E \int_0^\infty L(\eps(s^{-1}t))\one_{|\eps(s^{-1}t)|>r}\,ds\Bigg]\,,
      \quad t\geq 0\,,
    \end{multline}
    where the limit exists almost surely, $c\in\R$ is a constant, $L$
    is defined as in \eqref{eq:levy.fun} and $\{\eps_i,i\geq1\}$ is a
    sequence of \iid stochastic processes distributed as $\eps$, where
    $\eps$ \as takes values in $\DD'_0$,
    \begin{equation}
      \label{eq:eps-2}
      \E \int_0^\infty \min(1,\eps(t)^2)t^{-2}dt<\infty\,,
    \end{equation}
    and $\{\Gamma_i,i\geq1\}$ is the sequence of successive points of
    a homogeneous unit intensity Poisson process on $[0,\infty)$.
  \end{itemize}
\end{theorem}
\begin{proof}
  By Lemma~\ref{lemma:scaling}, a time-stable process can be
  alternatively described as an infinitely divisible stochastically
  continuous process whose L\'evy measure $Q$ satisfies
  \eqref{eq:q-scales}. It is obvious that $Q$ given by \eqref{eq:q-lp}
  satisfies \eqref{eq:q-scales}. It remains to show that the scaling
  property \eqref{eq:q-scales} yields \eqref{eq:q-lp}, so that (i)
  implies (ii).

  The following construction is motivated by the argument used to
  prove \cite[Th.~10.3]{evan:mol14}. By Lemma~\ref{lem:zero}, $Q$ is
  supported by $\DD'_0$.  Decompose $\DD'_0$ into the union of
  disjoint sets
  \begin{displaymath}
    \XX_0=\{f:\; \sup_{t\geq0} |f(t)|>1\}\,,
  \end{displaymath}
  and 
  \begin{displaymath}
    \XX_k=\{f:\; \sup_{t\geq0} |f(t)|\in (2^{-k},2^{-k+1}]\}\,,\quad k\geq 1\,.
  \end{displaymath}
  Recall the separant $\QQ$ and the exceptional set $\Omega_0$ from
  \eqref{eq:strongly-sep}. Since
  \begin{displaymath}
    \XX_0^c=\{f:\; |f(t)|\leq 1 \;, t\in[0,\infty)\}\,,
  \end{displaymath}
  we have that 
  \begin{displaymath}
    \XX_0\setminus \{f:\; \sup_{t\in \QQ} |f(t)|>1\}\subset\Omega_0\,.
  \end{displaymath}
  Similarly, 
  \begin{displaymath}
    \XX_k\setminus\{f:\; \sup_{t\in\QQ} |f(t)|\in (2^{-k},2^{-k+1}]\}
    \subset\Omega_0
  \end{displaymath}
  for all $k\geq1$. In view of the completeness assumption on the
  $\sigma$-algebra, all sets $\XX_k$, $k\geq 0$, are measurable.

  For each $k\geq0$, define the Borel map $\tau_k:\XX_k\to(0,\infty)$
  by
  \begin{displaymath}
    \tau_k(f)=\inf\{t>0:\; |f(t)|>2^{-k}\}\,,\quad  f\in \XX_k\,.
  \end{displaymath}
  Since all functions from $\DD'_0$ vanish at the origin, $\tau_k(f)$ is
  strictly positive and finite, and $\tau_k(c\circ f)=c^{-1}\tau_k(f)$ for
  all $c>0$. Let
  \begin{displaymath}
    \SS_k=\{f\in\XX_k:\; \tau_k(f)=1\}\,.
  \end{displaymath}
  Then $|f(1)|\geq 2^{-k}$ for all $f\in\SS_k$, $k\geq0$, and each
  function $g\in\XX_k$ can be uniquely represented as $s\circ f$ for
  $f\in\SS_k$ and $s>0$. The maps $(f,s)\mapsto s\circ f$ and
  $g\mapsto (\tau_k(g)\circ g,\tau_k(g)^{-1})$ are mutually inverse Borel
  bijections between $\SS_k\times(0,\infty)$ and $\XX_k$. This is seen
  by restricting $f$ onto a countable dense set $\QQ$ in
  $[0,\infty)$ and applying Lemma~\ref{lemma:measurab}. 

  Since $f\in\SS_k$ is right-continuous,
  \begin{displaymath}
    \Delta_k(f)=\sup\{t\in \QQ:\;
    |f(s)|>2^{-k-1}\;\text{for all}\; s\in [1,1+t]\}
  \end{displaymath}
  is strictly positive and measurable for each $k\geq0$. Define
  \begin{align*}
    \SS_{k0}&=\{f\in\SS_k:\; \Delta_k(f)>1\}\,,\\
    \SS_{kj}&=\{f\in\SS_k:\; \Delta_k(f)\in (2^{-j},2^{-j+1}]\}\,,\quad j\geq1\,.
  \end{align*}
  Then $\SS_k$ is the disjoint union of $\SS_{kj}$ for $j\geq0$ and
  $\XX_k$ is the disjoint union of 
  \begin{displaymath}
    \XX_{kj}=\{s\circ f:\; f\in\SS_{kj},\; s>0\}\,,\quad j\geq0\,.
  \end{displaymath}
  
  Fix any $k,j\geq0$. Then
  \begin{align*}
    q_{kj}&=Q(\{s\circ f:\; f\in\SS_{kj},\; s\in [1,1+2^{-j}]\})\\
    &\leq Q(\{f\in \DD'_0:\; |f(1)|\geq 2^{-k-1}\})\\ 
    &\leq 2^{2k+2}\int_{\{f:\; |f(1)|\geq 2^{-k-1}\}} \min(1,f(1)^2)Q(df)\\
    &\leq 2^{2k+2}\int \min(1,f(1)^2)Q(df)<\infty\,.
  \end{align*}
  By \eqref{eq:q-scales}, 
  \begin{align*}
    Q(\{s\circ f:\; & f\in\SS_{kj},\; s\geq 1\})\\
    &\leq \sum_{i=0}^\infty 
    Q(\{s\circ f:\; f\in\SS_{kj},\; s\in[(1+2^{-j})^i,\,
    (1+2^{-j})^{i+1}]\})\\
    &=\sum_{i=0}^\infty (1+2^{-j})^{-i} q_{kj}<\infty\,.    
  \end{align*}
  Thus, $Q$ restricted onto $\XX_{kj}$ is a push-forward under the map
  $(f,s)\to s\circ f$ of the product $\eta_{kj}\otimes\theta$ of a
  finite measure $\eta_{kj}$ supported by $\SS_{kj}$ and the measure
  $\theta$ on $(0,\infty)$ with density $s^{-2}ds$. Let $c_{kj}$ be
  some positive number, then the measure $\sigma_{kj}$ defined on
  $\DD'_0$ by $\sigma_{kj}(B)=c_{kj}\eta_{kj}(c_{kj}^{-1}\circ B)$
  assigns all its mass to the set $c_{kj}\circ \SS_{kj}$. Then the
  push-forward of $\sigma_{kj}\otimes\theta$ under the map
  $(f,s)\to s\circ f$ is $Q$ restricted on $\XX_{kj}$ and the total
  mass of $\sigma_{kj}$ equals $c_{kj}\eta_{kj}(\SS_{kj})$. By
  choosing $c_{kj}$ appropriately, it is always possible to achieve
  that $\sigma=\sum_{k,j\geq0} \sigma_{kj}$ is a probability measure
  on $\DD'_0$. Combining the push-forward representations of $Q$
  restricted to $\XX_{kj}$, $k,j\geq0$, we see that $Q$ is the
  push-forward of $\sigma\otimes\theta$ and so \eqref{eq:q-lp}
  holds. Given \eqref{eq:q-lp}, \eqref{eq:5} is equivalent to
  \eqref{eq:mu-condition}.

  The equivalence of (ii) and (iii) is immediate by choosing $\eps$ to
  be \iid with distribution $\sigma$ and noticing that \eqref{eq:5} is
  equivalent to \eqref{eq:eps-2} and that the limit in
  \eqref{eq:lepage-compensated} corresponds to the limit in
  \eqref{eq:compensated-sum}. 
\end{proof}

\begin{rem}
  There are many probability measure $\sigma$ that satisfy
  \eqref{eq:q-lp}, and so the distribution of $\eps$ in
  \eqref{eq:lepage-compensated} is not unique. For example, it is
  possible to scale the arguments of $\{\eps_i,i\geq1\}$ by a sequence
  of \iid positive random variables of mean one. The measure $\sigma$
  constructed in the proof of Theorem~\ref{thr:lp-stab} is supported
  by the set $\SS\subset\DD'$ such that if $f$ and $c\circ f$ with
  $c>0$ both belong to $\SS$, then $c=1$. The distribution of $\eps$
  is unique if $\eps$ is supported by a measurable set
  $\SS'\subset\DD'_0$ such that each $f\in\DD'_0$ can be uniquely
  represented as $c\circ g$ for $c>0$ and $g\in\SS'$. 
\end{rem}

\begin{rem}
  It follows from \cite[Th.~3.1]{bas:ros13} that the LePage
  series \eqref{eq:lepage-compensated} converges uniformly for $t$
  from any compact subset of $(0,\infty)$. If $H(t,r,V)=\eps(t/r)$,
  then Condition~(3.3) of \cite{bas:ros13} becomes
  \begin{displaymath}
    \int_0^\infty \P\{(\eps(t_1/r),\dots,\eps(t_k/r))\in B\}dr
    =Q(\{f:\;(f(t_1),\dots,f(t_k))\in B\})
  \end{displaymath}
  for all Borel $B$ in $\R^k\setminus\{0\}$, $t_1,\dots,t_k$, and
  $k\geq1$.
\end{rem}

\begin{theorem}
  \label{th:lepage-conv}
  A stochastically continuous stochastic process $\xi$ is time-stable
  without a Gaussian part and with the L\'evy measure $Q$ supported by
  $\DD'$ and satisfying \eqref{eq:mu-condition-1} if and only if
  \begin{equation}
    \label{eq:lepage}
    \xi(t)\eqd ct+\sum_{i=1}^\infty \eps_i(\Gamma_i^{-1}t)\,,\quad t\geq 0\,,
  \end{equation}
  where the series converges almost surely, $c$ is a constant,
  $\{\eps_i,i\geq1\}$ is a sequence of \iid stochastic processes with
  realisations in $\DD'_0$ such that
  \begin{equation}
    \label{eq:eps-1}
    \E \int_0^\infty \min(1,|\eps(t)|)t^{-2}dt<\infty\,,
  \end{equation}
  and $\{\Gamma_i, i\geq1\}$ is the sequence of successive points of
  the homogeneous unit intensity Poisson process on $[0,\infty)$.
\end{theorem}
\begin{proof}
  It suffices to note that \eqref{eq:eps-1} is equivalent to
  \eqref{eq:mu-condition-1}.
\end{proof}

\begin{rem}
  \label{rem:finite-int}
  Condition \eqref{eq:eps-1} (respectively \eqref{eq:eps-2}) holds if
  $\int_0^1\E|\eps(t)|t^{-2}dt<\infty$ (respectively
  $\int_0^1\E(\eps(t)^2)t^{-2}dt<\infty$). 
\end{rem}

\begin{rem}
  The time-stable process given by \eqref{eq:lepage} has the cumulant
  \begin{displaymath}
    \Psi(\lambda)=-\imath\lambda c
    + \int_0^\infty \E(1-e^{\imath \lambda \eps(s)})s^{-2} ds\,.
  \end{displaymath}
\end{rem}

\begin{corollary}
  \label{cor:non-neg}
  Each \as non-negative \cadlag time-stable process admits the LePage
  representation \eqref{eq:lepage}.
\end{corollary}

\begin{rem}
  \label{rem:rd}
  Analogues of the above results hold for time-stable processes with
  values in $\R^d$.  This can be shown by replacing $\SS_{kj}$ from
  the proof of Theorem~\ref{thr:lp-stab} with the Cartesian product of
  $d$-tuples of such sets
  $\SS^1_{k_1j_1}\times\cdots\times\SS^d_{k_dj_d}$, $k_i,j_i\geq 0$,
  $i=1,\dots,d$, constructed for each of the coordinates of the
  process. In particular, Corollary~\ref{cor:non-neg} applies for
  time-stable processes with values in $\R_+^d$.
\end{rem}

\begin{example}[L\'evy processes]
  \label{ex:levy-step}
  The spectral representation of a L\'evy process $\xi$ without a
  Gaussian part can be obtained by setting $f_i(t)=m_i\one_{t\geq
    \tau_i}$, where $\{(\tau_i,m_i),\; i\geq1\}$ is a marked Poisson
  process on $(0,\infty)\times (\R\setminus\{0\})$ with intensity
  measure being the product of the Lebesgue measure and a L\'evy
  measure $\Lambda$ on $\R\setminus\{0\}$. Indeed, then
  \begin{displaymath}
    \xi(t)\eqd ct + \lim_{r\downarrow0} 
    \left[\sum_{|x_i|>r} x_i\one_{t_i\leq t}-t\int_{|x|>r} 
      L(x)\Lambda(dx)\right]\,,
  \end{displaymath}
  which is the classical decomposition of a L\'evy process. In view of
  the uniqueness of the minimal spectral representation, the L\'evy
  measure $Q$ is supported by step functions of the type
  $m\one_{t\geq \tau}$.
  By Theorem~\ref{thr:lp-stab}, $\xi$ admits the series
  decomposition \eqref{eq:lepage-compensated} with
  $\eps(t)=\eta\one_{t\zeta\geq 1}$, where \eqref{eq:eps-2}
  corresponds to $\E[\min(1,\eta^2)\zeta]<\infty$. Following the
  construction from the proof of Theorem~\ref{thr:lp-stab}, the joint
  distribution of $(\eta,\zeta)$ can be constructed as follows.
  Denote $B_0=\{x\in\R:\; |x|>1\}$ and $B_k=\{x\in\R:\; 2^{-k}<|x|\leq
  2^{-k+1}\}$, $k\geq1$, let $q_k=\Lambda(B_k)$, $k\geq0$, and choose
  strictly positive $\{c_k,k\geq0\}$ such that $\sum_{k=0}^\infty
  c_kq_k=1$.  Then, for every Borel $A\subset \R\setminus\{0\}$, 
  \begin{displaymath}
    \P\{\eta\in A,\zeta=c_k^{-1}\}=\Lambda(A\cap
    B_k)c_k\quad \text{for every } k \geq 0.
  \end{displaymath}
  It is easy to see that
  \begin{displaymath}
    \E[\min(1,\eta^2)\zeta]=\int_{\R\setminus\{0\}}\min(1,x^2)\Lambda(dx).
  \end{displaymath}
  If $\xi$ has bounded variation, then Theorem~\ref{th:lepage-conv}
  applies and
  \begin{displaymath}
    \xi(t)\eqd ct+\sum_{i=1}^\infty \eta_i\Ind_{t\zeta_i\geq \Gamma_i}
  \end{displaymath}
  provides a LePage representation of $\xi$, cf.  \cite{ros01l} for
  the LePage representation of L\'evy processes on $[0,1]$.  The
  choice of $\eps(t)=\eta \one_{t \geq 1}$, $t\geq0$, yields the
  compound Poisson process $\xi(t)$, which becomes the standard
  Poisson process if $\eta=1$ \as 
  
  Note that the time and the size of the jump of $\eps$ may be
  dependent. For instance, let $\eps(t)=\eta\one_{t\geq\eta}$ for a
  positive random variable $\eta$. This random function always
  satisfies \eqref{eq:eps-1} and yields the L\'evy process
  \begin{displaymath}
    \xi(t)=\sum_{i=1}^\infty \eta_i\one_{t\geq \Gamma_i\eta_i}
  \end{displaymath}
  with the cumulant
  $\Psi(\lambda)=\E[(1-e^{\imath\lambda\eta})\eta^{-1}]$.
\end{example}

\begin{example}
  If $\eps(t)=\eta t^{1/\alpha}$, where $\alpha\in(0,2)$ and $\eta$ is
  a symmetric random variable with $\E|\eta|^\alpha<\infty$, then the
  LePage series \eqref{eq:lepage} converges \as by
  \cite[Th.~1.4.2]{ST} and $\xi(t)=ct + \zeta t^{1/\alpha}$ for a
  symmetric $\alpha$-stable random variable $\zeta$, see
  Example~\ref{stastex2}. In case $\alpha<1$, the symmetry of $\eta$
  is not required for the convergence of the LePage series and $\zeta$
  is strictly $\alpha$-stable by \cite[Th.~1.4.5]{ST}.
\end{example}

\begin{example}
  Choosing $\eps$ to be a stochastic process with stationary
  increments yields examples of time-stable processes with stationary
  increments which are not L\'evy processes. For instance, let $\eps$
  be the fractional Brownian motion with the Hurst parameter
  $H\in(\thf,1)$. Then \eqref{eq:eps-2} holds, since
  \begin{displaymath}
    \E\int_0^1 \min(1,\eps(t)^2)t^{-2}dt\leq 
    \int_0^1 \E\eps(t)^2 t^{-2}dt=\int_0^1 t^{2H-2}dt<\infty\,.
  \end{displaymath}
\end{example}

\begin{example}
  Let $\eps(t)=\xi(t^{1/\alpha})$, $t\geq0$, for $\alpha\in(0,1)$ and
  a time-stable process $\xi$ such that $\E|\xi(1)|<\infty$, so that
  \eqref{eq:eps-1} holds, since
  \begin{displaymath}
    \int_0^1 \E|\xi(t^{1/\alpha})|t^{-2}dt
    =  \E|\xi(1)|\int_0^1 t^{1/\alpha-2}dt<\infty.
  \end{displaymath}
  By conditioning on $\{\Gamma_i\}$ and using
  Proposition~\ref{prop:equiv-ts}, one obtains that
  \begin{displaymath}
    \sum_{i=1}^\infty \eps_i(\Gamma_i^{-1}t)
    =\sum_{i=1}^\infty \xi_i(\Gamma_i^{-1/\alpha} t^{1/\alpha})
    \eqd \xi(t^{1/\alpha}\sum_{i=1}^\infty \Gamma_i^{-1/\alpha})
    =\xi(t^{1/\alpha}\zeta)
  \end{displaymath}
  for a strictly $\alpha$-stable non-negative random variable $\zeta$
  independent of $\xi$. In this case the LePage series yields a
  representation of a sub-stable process from
  Example~\ref{ex:sub-stable}.
\end{example}

\begin{example}
  Let $\xi$ be a non-decreasing time-stable process that admits the
  LePage representation \eqref{eq:lepage} with $c=0$. If
  $\{X_i,i\geq1\}$ are \iid copies of a L\'evy process $X$ independent
  of $\xi$, then
  \begin{displaymath}
    \sum_{i=1}^\infty X_i(\eps_i(\Gamma_i^{-1}t))
  \end{displaymath}
  is the LePage representation of the time-stable process $X(\xi(t))$
  from Example~\ref{ex:subord}. This is seen by conditioning upon
  $\eps_i$ and $\Gamma_i$, $i\geq1$, and noticing that $X$ is
  stochastically continuous.
\end{example}

\begin{example}
  Consider $\eps(t)=(t^\beta-1)_+$, \ie the positive part of
  $(t^\beta-1)$, where $\beta\geq1$. If $\beta=1$, then the graph of
  $\xi$ is a continuous convex broken line with vertices at $(0,0)$
  and at
  \begin{displaymath}
    (\Gamma_n,\Gamma_n\sum_{i=1}^n \Gamma_i^{-1}-n)\,,\quad n\geq1\,.
  \end{displaymath}
  In order to obtain a differentiable curve, it is possible to use
  $\eps(t)=(t-1)_+^\beta$ for $\beta>1$.
\end{example}

\section{Time-stable step functions}
\label{sec:pure-jump-time}

Assume that $\xi$ is a \emph{pure jump} time-stable process, \ie its
paths are \cadlag piecewise constant functions with finitely many
jumps in each bounded Borel subset of $[0,\infty)$ and \as vanishing
at zero. In view of the assumed stochastic continuity and
\cite[Lemma~1.6.2]{sil04}, the jump times of $\xi$ have non-atomic
distributions.  The pure jump part of any \cadlag
time-stable process is also time-stable. This applies to the process
of jumps larger than $\delta>0$ in absolute value by noticing that the
jump part of the sum of two independent stochastic processes with
non-atomic distribution of jump times is equal to the sum of their jump
parts.

Let $\sM((0,\infty)\times\R)$ denote the family of marked point
configurations on $(0,\infty)$ with marks from $\R$. A \emph{marked
  point process} is a random element in $\sM((0,\infty)\times\R)$, see
\cite[Sec.~6.4]{DVJ1}. The successive ordered jump times
$\{\tau_k\}$ and the jump heights $\{m_k\}$ of a pure jump time-stable
process $\xi$ form the marked point process
$M=\{(\tau_k,m_k), k\geq1\}$, so that
\begin{displaymath}
  \xi(t)=\sum_{\tau_k\leq t}m_k\,,\qquad t\geq0.
\end{displaymath}
The sum is finite for every $t$, since the process is
assumed to have only a finite number of jumps in any bounded
interval. This construction introduces a correspondence between pure-jump
processes and marked point processes. Note that $M$ is a random closed
(and locally finite) set in $(0,\infty)\times\R$, see \cite{mo1}. If
convenient, we treat $M$ as a random counting measure, so that
$M(A\times B)$ equals the number of pairs $(\tau_k,m_k)$ that lie in
$A\times B$ for Borel $A\subset(0,\infty)$ and $B\subset\R$. Note that
$M(A\times\R)<\infty$ for all bounded Borel $A\subset(0,\infty)$.  The
process $\xi$ is compound Poisson if and only if $M$ is an
independently marked homogeneous Poisson process, \ie the jump times
form a homogenous Poisson process on $(0,\infty)$, while the jump
sizes are \iid random variables independent of the jump times.

\begin{proposition}
  \label{prop:fj}
  The time of the first jump of a \cadlag time-stable process, as well
  as the time of the first jump of size at least a given $\delta>0$
  have an exponential distribution.
\end{proposition}
\begin{proof}
  Observe that the time of the first jump of the sum of $n$
  independent processes equals the minimum of the first jump times
  $\tau_1,\dots,\tau_n$ of all summands. Then \eqref{eq:ts-def} yields
  that $n^{-1}\tau$ has the same distribution as the minimum of $n$
  \iid copies of $\tau$ and so characterises the exponential
  distribution. The same argument applies to the first jump of size at
  least $\delta>0$.
\end{proof}

The time of the second jump is not necessarily distributed
as the sum of two independent exponential random variables, since the
times between jumps may be dependent and even the gap between the
first and the second jump is no longer exponentially distributed in
general. 

\begin{example}
  \label{ex:subordination-jump}
  Let $N$ be the standard Poisson process, and let $\xi$ be an
  independent increasing time-stable process. Then $N(\xi(t))$ is a
  pure jump time-stable process, see Example~\ref{ex:subord}. In
  particular, the time $\tau$ of its first jump is exponentially
  distributed and $\tau=\inf\{t:\; \xi(t)\geq \zeta\}$ for the
  exponentially distributed random variable $\zeta$ independent of
  $\xi$.  
\end{example}

Scaling the argument of a pure jump process $\xi$ can be rephrased in
terms of scaling the marked point process $M$ corresponding to $\xi$, so
that $a\circ\xi$ corresponds to the marked point process
\begin{displaymath}
  a^{-1}\circ M=\{(a^{-1}\tau_k,m_k):\; k\geq1\})\,.
\end{displaymath}
The sum of independent pure jump processes corresponds to the
superposition of the corresponding marked point processes. The next
result relates the time stability property to the union-stability of
random sets (see \cite[Sec.~4.1.3]{mo1}); it immediately follows from
\eqref{eq:ts-def}.

\begin{proposition}
  \label{prop:ts-mp}
  A stochastically continuous pure jump process $\xi$ is time-stable
  if and only if its corresponding marked point process $M$ is a
  union-stable random closed set in the sense that
  \begin{equation}
    \label{eq:un-stab}
    M_1 \cup \cdots\cup M_n \eqd n^{-1}\circ M
  \end{equation}
  for each $n\geq2$, where $M_1,\dots,M_n$ are independent copies
  of $M$.
\end{proposition}

\begin{corollary}
  \label{cor:jump-plus}
  A stochastically continuous pure jump process $\xi$ is time-stable
  if and only if $\xi=\xi_+-\xi_-$ for the pair of stochastically
  continuous pure jump processes $(\xi_+,\xi_-)$ that form a pure jump
  time-stable process with values in $\R_+^2$.
\end{corollary}
\begin{proof}
  For $(\tau,m)\in(0,\infty)\times\R$ let $f(\tau,m)=(\tau,m_+,m_-)$,
  with $m_+$ and $m_-$ being the positive and negative parts of
  $m\in\R$. Then $M$ satisfies \eqref{eq:un-stab} if and only if
  $f(M)$ satisfies the analogue of \eqref{eq:un-stab} with the scaling
  along the first coordinate. Finally, this property of $f(M)$ is a
  reformulation of the time-stability of $(\xi_+,\xi_-)$, where
  $\xi_+$ is the sum of all positive jumps of $\xi$ and $\xi_-$ is the
  sum of all negative jumps.
\end{proof}

\begin{theorem}
  \label{thr:pj-lp}
  A stochastically continuous pure jump process $\xi$ is time-stable
  if and only if
  \begin{equation}
    \label{eq:pj-lepage}
    \xi(t)\eqd \sum_{i=1}^\infty \eps_i(\Gamma_i^{-1}t),
  \end{equation}
  where $\{\Gamma_i\}$ form a homogeneous unit intensity Poisson point
  process on $(0,\infty)$, and $\{\eps_i,i\geq1\}$ are independent
  copies of a random step function $\eps$ which is independent of
  $\{\Gamma_i\}$ and satisfies \eqref{eq:eps-1}.
\end{theorem}
\begin{proof}
  Sufficiency is immediate and follows from
  Theorem~\ref{th:lepage-conv}. For the necessity, consider the map
  $f$ from the proof of Corollary~\ref{cor:jump-plus} and note that
  $f(M)$ is an infinitely divisible point process on
  $(0,\infty)\times\R_+^2$. It is well known (see
  e.g. \cite[Th.~10.2.V]{DVJ2}) that such infinitely divisible marked
  point process can be represented as a superposition of point
  configurations that build a Poisson point process on
  $\sM((0,\infty)\times\R_+^2)$. The unique intensity measure
  $\tilde{Q}$ of this Poisson process is called the KLM measure of
  $M$. This measure can also be viewed as the L\'evy measure, see
  \cite[Cor.~6.9]{DMZ}.

  Each point configuration from $\sM((0,\infty)\times\R_+^2)$
  corresponds to a pure jump function. The push-forward of $\tilde{Q}$
  under this correspondence is the L\'evy measure of $(\xi_+,\xi_-)$
  that is supported by pure jump (and so \cadlag) functions. Since the
  components of $(\xi_+,\xi_-)$ are non-negative, Remark~\ref{rem:rd}
  yields its representation as
  \begin{displaymath}
    (\xi_+(t),\xi_-(t))\eqd
    \sum_{i=1}^\infty(\eps'_i(\Gamma_i^{-1}t),\eps''_i(\Gamma_i^{-1}t))\,,
  \end{displaymath}
  so that $\xi$ admits the series representation \eqref{eq:pj-lepage}
  with $\eps=\eps'-\eps''$. 
\end{proof}

\begin{rem}
  In the classical LePage series for random vectors, it is possible to
  scale the directional component to bring its norm to $1$.  However,
  it is not possible in general to rescale the argument of $\{\eps_i,
  i\geq1\}$ from \eqref{eq:pj-lepage} in order to ensure that each
  function has the first jump at time $1$. 
\end{rem}

\begin{rem}
  It is possible to derive Theorem~\ref{thr:pj-lp} from the LePage
  representation of the marked point process $M$ as the union of
  clusters corresponding to the Poisson cluster process determined by
  $\tilde{Q}$. The corresponding series representation then becomes
  \begin{displaymath}
    M=\bigcup_{i=1}^\infty \Gamma_i\circ E_i,
  \end{displaymath}
  where $\{E_i,i\geq1\}$ is a point process on
  $\sM((0,\infty)\times\R_+^2)$ with the intensity measure
  $\tilde{Q}$. 
\end{rem}

If $\eps$ has a single jump only, then \eqref{eq:pj-lepage} yields a
L\'evy process, see Example~\ref{ex:levy-step}. 

\begin{example}
  \label{ex:int-part}
  Let $\eps(t)=[t]$ be the integer part of $t$. Then 
  \begin{displaymath}
    \xi(t)=\sum_{k=1}^\infty N(t/k)\,,
  \end{displaymath}
  where $N(t)$ is the Poisson process. For every $t\geq 0$, the series
  consists of a finite number of summands and so converges almost
  surely. Note that $\xi(t)$ is not integrable for $t>0$. The jump
  sizes of $\xi$ are always one, and the jump times form a point
  process on $\R_+$ obtained as the superposition of the set of
  natural numbers scaled by $\Gamma_i$, $i\geq1$.
\end{example}

\section*{Acknowledgments}
\label{sec:achnowledgements}

The authors are grateful to Zakhar Kabluchko and Rolf Riedi for
discussions and hints at various stages of this work. The second
author is grateful to Steve Evans for discussions concerning general
LePage series. The second author was partially supported by Swiss
National Science Foundation grant 200021-153597.

\end{document}